\documentclass[12pt]{amsart}
\usepackage{amsmath}
\usepackage{amssymb}
\usepackage[mathcal]{eucal}
\usepackage[all]{xy}
\usepackage{latexsym}
\usepackage{amstext}
\usepackage{amsfonts}
\usepackage{amsthm}
\usepackage{amsopn}
\usepackage{amsbsy}
\usepackage{layout}
\usepackage{color}
\usepackage{graphicx}
\usepackage{bm}
\usepackage{yfonts}
\usepackage{pifont}
\usepackage[scr]{rsfso}
\usepackage{tikz}
\newtheorem{thm}{Theorem}[section]

\newtheorem{cor}[thm]{Corollary}
\newtheorem{pro}[thm]{Proposition}
\newtheorem{ex}[thm]{Example}

\newtheorem{claim}{Claim}

\newcommand{\Int}{\mbox{{\rm Int}}}

\def\supp{{\mathrm {supp}}\,}

\makeatletter
\def\Int{\mathop{\operator@font Int}\nolimits}
\makeatother

\begin{document}

\title[Bounded uniform homeomorphisms between $C_p^*$-spaces preserve pseudocompactness]
{Bounded uniform homeomorphisms between $C_p^*$-spaces preserve pseudocompactness}

\author{M. Krupski}
\address{Institute of Mathematics, University of Warsaw,
ul. Banacha 2, 02 097 Warszawa, Poland}
\email{mkrupski@mimuw.edu.pl}

\author{V. Valov}
\address{Department of Computer Science and Mathematics, Nipissing University,
100 College Drive, P.O. Box 5002, North Bay, ON, P1B 8L7, Canada}
\email{veskov@nipissingu.ca}
\thanks{The first author was partially supported by the NCN (National Science Centre, Poland) research Grant no. 2020/37/B/ST1/02613. The second author was partially supported by NSERC Grant 2025-07173}


\begin{abstract}
For any Tychonoff space $X$, let $C_p(X)$ (resp., $C^*_p(X)$) be the set of all continuous real-valued (resp., and bounded) functions on $X$ with the pointwise convergence topology.  
Given Tychonoff spaces $X$ and $Y$, Uspenskij \cite{us} proved that if $C_p(X)$ is uniformly homeomorphic to $C_p(Y)$, then $X$ is pseudocompact if and only if $Y$ is pseudocompact. 
The second author and Vuma \cite{valvu} have shown that linear homeomorphisms between $C_p^*(X)$ and $C_p^*(Y)$ also preserve pseudocompactness. 
Recently Baars-van Mill-Tkachuk \cite{bmt} gave another proof of that result and raised the question if the same remains true provided $C_p^*(X)$ and $C_p^*(Y)$ are uniformly homeomorphic. 
In the present paper we introduce the notion of bounded uniformly continuous maps and show that every bounded uniform homeomorphism between $C_p^*(X)$ and $C_p^*(Y)$ preserve pseudocompactness. Moreover, we show that if there is a uniform homeomorphism $\varphi:C_p^*(X)\to C_p^*(Y)$, where $X$ is pseudocompact and $Y$ is normal, then $Y$ is $\sigma$-pseudocompact, i.e. $Y$ is a countable union of closed pseudocompact subsets. It is also shown that a continuous linear map between $C_p$-spaces is norm-bounded if and only if it is bounded in our sense.
\end{abstract}

\maketitle\markboth{}{Uniform homeomorphisms}


\section{Introduction}\label{intro}
For a Tychonoff space $X$,
by $C_p(X)$ we denote the space of all continuous real-valued functions on $X$ equipped with the pointwise convergence topology, while
$C^*_p(X)$ is the subspace of $C_p(X)$ consisting of the bounded functions.
Given Tychonoff spaces $X$ and $Y$, Uspenskij \cite{us} proved that if $C_p(X)$ is uniformly homeomorphic to $C_p(Y)$, then $X$ is pseudocompact if and only if $Y$ is pseudocompact (see G\'orak-Krupski-Marciszewski \cite{gkm} for another proof of this result). Valov-Vuma \cite{valvu} have shown that
pseudocompactness is preserved by linear homeomorphisms between $C_p^*(X)$ and $C_p^*(Y)$.
Recently, Baars-van Mill-Tkachuk \cite{bmt} reestablished the same result and raised the question if this remains true for uniform homeomorphisms. 
It is well known that every linear continuous map between $C_p^*(X)$ and $C_p^*(Y)$ is norm-bounded.
In the present paper we introduce an analogue of the norm-boundedness for uniformly continuous maps between $C_p$-spaces and prove
the following (here $D_p(Y)$ is either $C_p(Y)$ or $C_p^*(Y)$):
\begin{thm}\label{main}
Let $\varphi: C_p^*(X)\to D_p(Y)$ be a uniform homeomorphism and $X$ be pseudocompact: Then 
\begin{itemize}
\item[(i)] $Y$ is $\sigma$-pseudocompact provided $Y$ is normal; 
\item[(ii)]  $Y$ is pseudocompact provided $\varphi$ is a bounded homeomorphism.
\end{itemize}
\end{thm} 
Recall that if $\kappa$ is an infinite cardinal, a space $X$ is called {\em $\kappa$-pseudocompact} \cite{ken} if for every continuous map $f:X\to\mathbb R^\kappa$, the image $f(X)$ is compact.
\begin{cor}
Let $\varphi: C_p^*(X)\to C_p^*(Y)$ be a uniform homeomorphism such that both $\varphi$ and $\varphi^{-1}$ are bounded. Then: 
\begin{itemize}
\item[(i)] $X$ is pseudocompact if and only if $Y$ is pseudocompact;
\item[(ii)] For every infinite cardinal $\kappa$, the space $X$ is $\kappa$-pseudocompact if and only if $Y$ is $\kappa$-pseudocompact.
\end{itemize}
\end{cor}

Theorem \ref{main} together with the Uspenskij's \cite{us} characterization of pseudocompactness implies also that if  $\varphi: C_p(X)\to C_p^*(Y)$ is a bounded uniform homeomorphism, then $X$ is pseudocompact if and only if $Y$ is pseudocompact.

\section{Preliminary results}
One of the main tools used to prove
Theorem \ref{main} is the Gul'ko's \cite{gu} notion of supports, see also \cite{arbit}, \cite{ev}, \cite{k}, \cite{mp}.
For our purposes, we consider the following version of Gul'ko's support, see \cite{ev}, \cite{k}.
Suppose $\varphi:D_p(X)\to D_p(Y)$ is a  uniformly continuous surjection. For every $y\in Y$
and a set $K\subset X$ we define
$$a_\varphi(y,K)=\sup\{|\varphi(f)(y)-\varphi(g)(y)|:f,g\in D_p(X){~}\hbox{and}{~}| f(x)- g(x)|<1{~} \forall x\in K\}.$$

For every $y\in Y$ and $n\in\mathbb N$ consider the following families :
$$\mathcal A(y)=\{K\subset X:K{~}\hbox{is finite and}{~}a_\varphi(y,K)<\infty\},$$
$$\mathcal A_n(y)=\{K\subset X:K{~}\hbox{is finite and}{~}a_\varphi(y,K)\leq n\}.$$
Let $Y_n=\{y\in Y:\mathcal A_n(y)\neq\varnothing\}$,
$Y_{n,k}=\{y\in Y_n:\exists K\in\mathcal A_n(y){~}\hbox{with}{~}|K|\leq k\}$,
$K(y)=\bigcap\mathcal A(y)$ and $a_\varphi(y)=a_\varphi(y,K(y))$. We also consider the sets $M(n,1)=Y_{n,1}$ and $M(n,k)=Y_{n,k}\backslash Y_{2n,k-1}$ for $k\geq 2$. 
The next proposition provides the key property of the Gul'ko's supports $K(y)$ (see, e.g. \cite{mp}).
\begin{pro}\label{properties of Gulko support}
For every $y\in Y$, we have $K(y)\in\mathcal{A}(y)$ and $K(y)$ is the minimal, with respect to inclusion, set in $\mathcal{A}(y)$.
Moreover, $Y=\bigcup_{n\geq 1}Y_{n}$.
\end{pro}

We say that a uniformly continuous surjection $\varphi:D_p(X)\to D_p(Y)$ is {\em bounded} if there is $m$ such that $Y=Y_m$.
This means that for every $y\in Y$ there is a set $K_m(y)\in\mathcal A_m(y)$. 
\begin{pro}\label{Y_m+1 is nonempty}
Suppose that $\varphi:D_p(X)\to D_p(Y)$ is a uniformly continuous surjection. Let $y\in Y$. If  there is a subset $A\subset X$ satisfying $a_\varphi(y,A)\leq m$, then there is a finite subset $K\subseteq X$ with $a_\varphi(y,K)\leq m+1$, so $y\in Y_{m+1}$.
\end{pro}
\begin{proof}
Since $\varphi$ is uniformly continuous, there is $\varepsilon>0$ and a finite subset $K\subseteq X$ such that:
\begin{equation}\label{eq1}
\begin{split}
&\mbox{If } f,g\in D_p(X) \mbox{ satisfy } |f(x)-g(x)|<\varepsilon\; \mbox{for all} x\in K,\\ &\mbox{then } |\varphi(f)(y)-\varphi(g)(y)|<1.
\end{split}
\end{equation}
We will show that $K$ is as desired. To this end, fix $f_0,g_0\in D_p(X)$ such that $|f_0(x)-g_0(x)|<1$ for all $x\in K$. Let $c=\max\{|f_0(x)-g_0(x)|:x\in K\}$. Clearly, $c<1$. Let $s:K\to [-c,c]$ be the restriction of the function $f_0-g_0$ to $K$, and let $\tilde{s}:\beta X\to [-c,c]$ be the extension of $s$ over the \v{C}ech-Stone compactification $\beta X$ of $X$. Finally, let $h=\tilde{s}\upharpoonright X$ be the restriction of $\tilde{s}$ to $X$. For any $x\in X$ (in particular for any $x\in A$) we have:
$$|(h+g_0)(x)-g_0(x)|=|h(x)|\leq c <1.$$
So we infer from $a(y,A)\leq m$ that
\begin{equation}\label{eq2}
 |\varphi(h+g_0)(y)-\varphi(g_0)(y)|\leq m
\end{equation}
If $x\in K$, then
$$|(h+g_0)(x)-f_0(x)|=|s(x)+g_0(x)-f_0(x)|=0.$$
Hence, according to \eqref{eq1},
\begin{equation}\label{eq3}
|\varphi(h+g_0)(y)-\varphi(f_0)(y)|<1.
\end{equation}
From \eqref{eq2} and \eqref{eq3}, we get
\begin{equation*}
 \begin{split}
&|\varphi(f_0)(y)-\varphi(g_0)(y)|=|\varphi(h+g_0)(y)-\varphi(g_0)(y)-(\varphi(h+g_0)(y)-\varphi(f_0)(y))|\leq \\
&|\varphi(h+g_0)(y)-\varphi(g_0)(y)|+|(\varphi(h+g_0)(y)-\varphi(f_0)(y)|<m+1.
 \end{split}
\end{equation*}
Since $f_0$ and $g_0$ were chosen arbitrarily, the result follows.
\end{proof}

If $\varphi:D_p(X)\to D_p(Y)$ is a continuous and linear map, then for every $y\in Y$ we define $\supp(y)$ to be the set of all $x\in X$ such that for every neighborhood $U_x$ of $x$ in $X$ there is a function $f\in D_p(X)$ with $f(X\backslash U_x)=\{0\}$ and $\varphi(f)(y)\neq 0$. In case $\varphi(D_p(X))$ is dense in $D_p(Y)$, every $\supp(y)$ is a finite and non-empty subset of $X$. Moreover, there are non-zero numbers $\lambda_x$ for all $x\in\supp(y)$, such that $\varphi(f)(y)=\sum_{x\in\supp(y)}\lambda_x\cdot f(x)$ for all $f\in D_p(X)$, see \cite{ar1}, \cite{bd}.
\begin{pro}
Let $\varphi:D_p(X)\to D_p(Y)$ be a continuous and linear map such that $\varphi(D_p(X))$ is dense in $D_p(Y)$. Then $K(y)=\supp(y)$ and 
$\varphi$ is bounded in our sense if and only if it is norm-bounded.
\end{pro}

\begin{proof}
The representation of $\varphi(f)(y)$ as $\sum_{x\in\supp(y)}\lambda_x\cdot f(x)$ for all $f\in D_p(X)$
shows that $a_\varphi(y,\supp(y))=\sum_{x\in\supp(y)}|\lambda_x|<\infty$ for all $y\in Y$.
Hence, $\supp(y)\in\mathcal A(y)$ for all $y\in Y$, which implies $K(y)\subset\supp(y)$ because $K(y)$ is the minimal set in $\mathcal A(y)$
(see Proposition 2.1).
 To prove the reverse inclusion, suppose there are $y_0\in Y$ and $x_0\in\supp(y_0)\backslash K(y_0)$. Let $f\in D_p(X)$ be such that $f(x)=0$ for all $x\in (K(y_0)\cup\supp(y_0))\backslash\{x_0\}$ and $f(x_0)>a_\varphi(y_0,K(y_0))/|\lambda_{x_0}|$. Then $|\varphi(f)(y_0)|=|\lambda_{x_0}f(x_0)|>a_\varphi(y_0,K(y_0))$.
 By linearity, $\varphi$ maps the zero function onto the zero function, and since $f(x)=0$ for all $x\in K(y_0)$, we have $|\varphi(f)(y_0)|\leq a_\varphi(y_0,K(y_0))$, a contradiction. Therefore, $K(y)=\supp(y)$.

If $\varphi$ is norm-bounded, then $||\varphi||=\sup\{|\varphi(f)(y)|:||f||<1,y\in Y\}<\infty$. So, by Proposition \ref{Y_m+1 is nonempty}, if $m\geq ||\varphi||+1$, then
$Y\subset Y_m$ and $\varphi$ is bounded in our sense.
For the reverse implication, suppose $Y\subset Y_m$ for some $m$. Since
$||\varphi||=\sup_{y\in Y}a{_\varphi}(y,X)$ and $a{_\varphi}(y,X)\leq a{_\varphi}(y,K_m(y))$, we have $||\varphi||\leq m$.
\end{proof}
\begin{ex}
There exists a non-linear bounded uniformly continuous surjection $\varphi:D_p(X)\to D_p(Y)$.
\end{ex}

Indeed, let $T:D_p(X)\to D_p(Y)$ be a non-linear uniformly continuous surjection such that $X$ and $Y$ are metrizable spaces. Then, by \cite{gu} (see also \cite{mp}), the sets $Y_{n,k}=\{y\in Y:\exists K\subset X, a_T(y,K)\leq n,|K|\leq k\}$ are closed in $Y$ and $Y=\bigcup_{n,k}Y_{n,k}$. So, $Y_{n,k}\neq\varnothing$ for some integers $n,k$. Let $\theta: D_p(Y)\to D_p(Y_{n,k})$ be the restriction map $g\mapsto g|Y_{n,k}$ and $\varphi=\theta\circ T$.
Then $\varphi:D_p(X)\to D_p(Y_{n,k})$ is a non-linear uniformly continuous surjection. Moreover, for every $y\in Y_{n,k}$ there exists $K\in\mathcal A_n(y)$ such that $|K|\leq k$ and $a_T(y,K)\leq n$. On the other hand, $a_T(y,K)=a_\varphi(y,K)$. Therefore, $\varphi$ is bounded.

\medskip

It is well known that if $f:X\to Y$ is a surjective map, then the quotient topology on $Y$ generated by $f$ is not always completely regular. The strongest among all completely regular topologies on $Y$ with respect to which $f$ is still continuous exists and it called {\em $R$-quotient topology}, see \cite{ar}. The set $Y$ with the $R$-quotient topology is denoted by $Y^q$ and $q(f):X\to Y^q$ is the corresponding to $f$ continuous map. 

For a space $X$ by 
$\mathcal F_X$ we denote the family of all continuous maps from $X$ onto a second countable space and direct this family by the relation:
\begin{align*}
 g\succ f \mbox{  if there is a continuous map } h(g,f): g(X)\to f(X) \mbox{ with }
 h(g,f)\circ g=f.
\end{align*}

If $f\in\mathcal F_X$ let $X_f=f(X)$, $D_p(f)=f^\sharp(D_p(X_f))$ and $D_p^q(f)=q^\sharp(f)(D_p(X_f^q))$, where $f^\sharp:D_p(X_f)\to D_p(X)$ and
$q^\sharp(f):D_p(X_f^q)\to D_p(X)$ are the dual continuous maps of $f$ and $q(f)$.
We need the following observations:

\smallskip

\begin{itemize}\itemsep10pt
\item[Fact 1.] $D_p(X)=\bigcup\{D_p(f):f\in\mathcal F_X\}$ and $D_p(f)\subset D_p^q(f)$;
\item[Fact 2.] For every countable $A\subset D_p(X)$ there exists $f\in\mathcal F_X$ with $A\subset D_p(f)$;
\item[Fact 3.] $D_p^q(f)\subset D_p(X)$ is closed and separable;
\end{itemize}

\smallskip

The first two facts are obvious. To prove Fact 3, observe that $D_p^q(f)$ is closed in $D_p(X)$ because $q(f)$ is $R$-quotient \cite{ar}. On the other hand, there is a one-to-one map from $X_f^q$ onto $X_f$, so $D_p(f)$ is dense in $D_p^q(f)$. Moreover, $D_p(f)$ is homeomorphic to $D_p(X_f)$ and, since $X_f$ has countable weight, $D_p(X_f)$ is separable.
 
\section{Proof of Theorem $1.1$}
Let us recall that a subset $A$ of a space $X$ is \textit{$C$-embedded} ($C^*$-embedded) if for every $f\in C_p(A)$ (for every $f\in C_p^*(A)$) there exists $\hat{f}\in C_p(X)$ ($\hat{f}\in C_p^*(X)$) such that $\hat{f}\upharpoonright A = f$. We also say that a set $A\subset X$ is bounded in $X$ if $f(A)$ is a bounded set for every $f\in C(X)$. Obviously, $X$ is pseudocompact if and only it is bounded in itself. It is easily seen that if a set $A\subset X$ is not bounded in $X$, then $A$ contains a closed, discrete and $C$-embedded subset of $X$, see \cite{vm}.

\begin{pro}
Let $X=\bigcup_{i=1}^\infty X_i$ be a union of countably many metric compacta $X_i$ and let $\varphi: C_p^*(X)\to D_p(Y)$ be a uniformly continuous surjection. Then $Y$ is $\sigma$-compact. 
\end{pro}
\begin{proof} We can assume that the sequence $\{X_i\}$ is increasing. For every natural numbers $n,i$
let $H_{n,i}=\{y\in Y_n:\exists K\in\mathcal A_n(y){~}\hbox{with}{~}K\subset X_i\}$. Since the elements of every $\mathcal A_n(y)$ are finite subsets of $X$, $Y$ is the union of all $H_{n,i}$.
 
 \begin{claim}
The closure $\overline H_{n,i}$ in $Y$ of every $H_{n,i}$ is compact.
 \end{claim}
Since $X$ has a countable networkweight $nw(X)$ and $nw(X)=nw(C_p^*(X))=nw(D_p(Y))=nw(Y)$ (see \cite[Problem 172]{tk}), $Y$ is Lindel\"{o}f. Hence, all sets $\overline H_{n,i}$ are also Lindel\"{o}f as closed sets in $Y$. So it suffices to show that every $\overline H_{n,i}$ is bounded in $Y$. On the other hand, $H_{n,i}$ is bounded in $Y$ implies $\overline H_{n,i}$ is also bounded in $Y$. Therefore, the proof reduces to showing that every set $H_{n,i}$ is bounded in $Y$.
By homogeneity of $D_p(Y)$, we can assume that $\varphi$ maps the zero-function on $X$ to the zero-function on $Y$. Striving for a contradiction, suppose that $H_{n,i}$ is not bounded in $Y$ for some integers $n,i$. Then $H_{n,i}$ contains a closed discrete and $C$-embedded in $Y$ set $T=\{y_m:m=1,2,\ldots \}$. Then $T$ is also $C^*$ embedded in $Y$ and let $\exp(T)$ be the family of all subsets of $T$.
For $A\in \exp(T)$, define a function $h_A:T\to \mathbb{R}$ as follows:

\begin{align*}
h_A(y)=
\begin{cases}
n+1&\quad \text{for}\;y\in A\\
0&\quad \text{for}\;y\in T\setminus A
\end{cases}
\end{align*}
Since $T$ is $C^*$-embedded, for every $A\in \exp(T)$, the function $h_A$ extends to a function $g_A\in C_p^*(Y)$. For each $A\in \exp(T)$ choose
$f_A\in C_p^*(X)$ such that $\varphi(f_A)=g_A.$
The space $X_i$ is compact metric, so the space $C^*(X_i)$ of continuous real-valued functions on $X$ endowed with the supremum metric, is separable (see \cite[Proposition 7.6.2]{se}). Hence, since the family $\{f_A|X_i:A\in \exp(T)\}$ is uncountable, there exist $A,B\subset T$ such that $A\neq B$ and $|f_A(x)-f_B(x)|<1$ for all $x\in X_i$. Since for every $y_m$ there is $K_m\in\mathcal A_n(y_m)$, the last inequality implies that $|g_A(y_m)-g_B(y_m)|\leq a_\varphi(y_m,K_m)\leq n$. On the other hand, if $y\in A\bigtriangleup B$, then $|g_A(y)-g_B(y)|=n+1$ (here $A\bigtriangleup B$ is the symmetric difference of $A$ and $B$, which is nonempty because $A\neq B$). This contradiction completes the proof of Claim 1. 

Finally, the proof of Proposition 3.1 follows from Claim 1 and the equality $Y=\bigcup_{n,i}\overline H_{n,i}$.
\end{proof}

\begin{cor}
Let $\varphi: C_p^*(X)\to D_p(Y)$ be a bounded uniformly continuous surjection such that $X$ is a metric compactum. Then $Y$ is a metric compactum.
\end{cor}
\begin{proof}
Since $\varphi$ is bounded, there is $n$ such that $Y=Y_n$. Then repeating the arguments from the proof of Claim 1, we can show that $Y$ is bounded in itself, so $Y$ is pseudocompact. On the other hand, $Y$ is Lindel\"{o}f because it has a countable networkweight. So, $Y$ is a metric compactum. 
\end{proof}

{\em Proof of Theorem $1.1$.} Suppose $\varphi:C_p^*(X)\to D_p(Y)$ is a uniform homeomorphism and $X$ is pseudocompact. We assume that $D_p(Y)=C^*(Y)$, the other case is similar. If $Y$ is normal, we are going to show that the closure $\overline Y_m$ of every $Y_{m}\subset Y$ in $Y$ is pseudocompact. So, we need to prove that for every map $g_0:\overline Y_{m}\to\mathbb R^\omega$ the image $g_0(\overline Y_{m})$ is compact. We fix such a map $g_0$ and extend it to a map $g_1:Y\to\mathbb R^\omega$. This can be done because $Y$ is normal. We follow the proof of \cite[Theorem 4.3]{valvu}. 
Since, by Fact 3, $C_p^q(g_1)\subset C_p^*(Y)$ is closed and separable, Facts 1 and 2 together with \cite[S.163 (iii)]{tk} imply the existence of $f_1\in\mathcal F_X$ with
$\varphi^{-1}(C_p^q(g_1))\subset C_p^q(f_1)$. Again using Facts 1, 2 and 3 we find $g_2\in\mathcal F_Y$ such that $\varphi(C_p^q(f_1))\subset C_p^q(g_2)$. We can assume $g_2\succ g_1$ (otherwise instead of $g_2$ we take the diagonal product of $g_1$ and $g_2$). Proceeding in this way, we construct two sequences $\{f_n\}\subset\mathcal F_X$ and $\{g_n\}\subset\mathcal F_Y$ such that:
\begin{itemize}
\item[(1)] $f_{n+1}\succ f_n$ and $g_{n+1}\succ g_n$;
\item[(2)] $\varphi(C_p^q(f_n))\subset C_p^q(g_{n+1})$ and $\varphi^{-1}(C_p^q(g_n))\subset C_p^q(f_{n})$.
\end{itemize}
Let $f$ (resp., $g$) be the diagonal product of all $f_n$ (resp., all $g_n$). Then the union $\bigcup_{n\geq 1}h^\sharp(f,f_n)(C_p^*(f_n(X)))$ is dense in $C_p(X_f)$, where $h^\sharp(f,f_n):C_p^*(f_n(X))\to C_p^*(X_f)$ is the dual of the map $h(f,f_n):X_f\to f_n(X)$. Because $C_p(f)$ is dense in $C_p^q(f)$ (see \cite[S.163 (iv)]{tk}), we have that $\bigcup_{n\geq 1}C_p^q(f_n)$ is dense in $C_p^q(f)$. Similarly, $\bigcup_{n\geq 1}C_p^q(g_n)$ is dense in $C_p^q(g)$. The last two inclusions imply that $\varphi(C_p^q(f))=C_p^q(g)$ (recall that, by Fact 3, $C_p^q(f)$ is closed in $C_p^*(X)$ and $C_p^q(g)$ is closed in  $C_p^*(Y)$). Since both $X_f^q$ and $Y_g^q$ carry the $R$-quotient topology, $C_p^q(f)$ is linearly homeomorphic to $C_p^*(X_f^q)$ and $C_p^q(g)$ is linearly homeomorphic to $C_p^*(Y_g^q)$. Therefore, there is a uniform homeomorphism $\psi:C_p^*(X_f^q)\to C_p^*(Y_g^q)$. Because $X$ is pseudocompact, $X_f$ is homeomorphic to $X^q_f$, see \cite{ar}. So, $X^q_f$ is a metric compactum.

Now, for every $z\in X_f^q$ consider the families 

$$\mathcal B_k(z)=\{L\subset X_f^q:L{~}\hbox{is finite and}{~}a_{\psi}(z,L)\leq k\},$$
where 
$$a_{\psi}(z,L)=\sup\{|\psi(h)(z)-\psi(l)(z)|:h,l\in C^*(X_f^q){~}\hbox{and}{~}|h(x)-l(x)|<1{~} \forall x\in L\}.$$
Let also $(Y_g^q)_k=\{z\in Y^q_g:\mathcal B_k(z)\neq\varnothing\}$.
\begin{claim}
The set $q(g)(Y_m)$ is contained in the set $(Y_g^q)_m$.
\end{claim}
Indeed, for every $z\in q(g)(Y_m)$ fix $y_z\in Y_m$ with $q(g)(y_z)=z$. Since $y_z\in Y_m$, there exists a finite set $K\subset X$ with $K\in\mathcal A_m(y_z)$. We are going to show that $q(f)(K)\in\mathcal B_m(z)$. To this end, let $h_1,h_2\in C^*_p(X^q_f)$ such that $|h_1(u)-h_2(u)|<1$ for all $u\in q(f)(K)$. Let $\theta_i=h_i\circ q(f)$, for $i=1,2$. Then $|\theta_1(x)-\theta_2(x)|<1$ for all $x\in K$, so
$|\varphi(\theta_1)(y_z)-\varphi(\theta_2)(y_z)|\leq m$. Since $\varphi(\theta_i)=\psi(h_i)\circ q(g)$, we have
$\varphi(\theta_i)(y_z)=\psi(h_i)(z)$ for $i=1,2$. Hence, $|\psi(h_1)(z)-\psi(h_2)(z)|\leq m$. Therefore, $q(f)(K)\in\mathcal B_m(z)$, which implies the inclusion $q(g)(Y_m)\subset (Y_g^q)_m$.

\begin{claim}
The set $q(g)(Y_m)$ is bounded in $Y_g^q$.
\end{claim}
This follows from the proof of Claim 1. Indeed, if $q(g)(Y_m)$ were not bounded in $Y_g^q$, then $q(g)(Y_m)$ would contain a closed discrete and $C$-embedded subset $T=\{z_k\}_{k\geq 1}$ of $Y_g^q$. Then, as in the proof of Claim 1, for every $A\in \exp(T)$ there is a function $g_A\in C^*(Y_g^q)$ such that $g_A(z)=m+1$ for $z\in A$ and $g_A(z)=0$ for $z\in T\backslash A$. For each $A\in \exp(T)$ choose $h_A\in C^*_p(X_f^q)$ with $\psi(h_A)=g_A$. Because $X_f^q$ is a metric compactum, there exist $A,B\subset T$ such that $A\neq B$ and $|h_A(x)-h_B(x)|<1$ for all $x\in X_f^q$. Since for every $z_k$ there exists $L_k\in\mathcal B_m(z_k)$ with $L_k\subset X_f^q$, we have $|h_A(x)-h_B(x)|<1$ for all $x\in L_k$. This implies
$|g_A(z_k)-g_B(z_k)|\leq a_\psi(z_k,L_k)\leq m$. On the other hand, if $z\in A\bigtriangleup B$, then $|g_A(z)-g_B(z)|=m+1$. This completes the proof of Claim 3.

Now, we can complete the proof of Theorem 1.1(i). According to Claim 3, the set $q(g)(Y_m)$ is bounded in $Y_g^q$. This implies that  $g(Y_m)$ is bounded in $Y_g$ and since $Y_g$ is a metric space, $g(Y_m)$ is compact. Consequently, $g(\overline Y_m)=g(Y_m)$ is compact. Since $g\succ g_1$, there is a continuous map from $g(\overline Y_m)$ onto $g_1(\overline Y_m)$. Hence,
$g_1(\overline Y_m)$ is also compact. Finally, observe that $g_0(\overline Y_m)=g_1(\overline Y_m)$. Therefore, $\overline Y_m$ is pseudocompact for every $m$. So, $Y$ is $\sigma$-pseudocompact.

To prove Theorem 1.1(ii), we need to show that $g_1(Y)$ is compact for every $g_1\in \mathcal F_Y$. So, we fix such $g_1$ and repeat the above arguments to
find maps $f\in\mathcal F_X$, $g\in\mathcal F_Y$ and a uniform homeomorphism $\psi:C_p^*(X_f^q)\to C_p^*(Y_g^q)$.  
\begin{claim}
The uniform homeomorphism $\psi$ is bounded.
\end{claim}
The proof of Claim 4 follows from Claim 2. Indeed, since $\varphi$ is a bounded homeomorphism, there is $m$ such that $Y=Y_m$. Then, $q(g)(Y_m)=Y_g^q$ and, by Claim 2, $Y_g^q=(Y_g^q)_m$. Hence, $\psi$ is a bounded homeomorphism.

Because $X_f$ is a metric compactum, by Corollary 3.2, $Y^q_g$ is compact, and hence homeomorphic to $Y_g$ (see \cite{ar}). Finally, since $g\succ g_1$, the set $g_1(Y)$ is also compact. This completes the proof of Theorem 1.1. $\Box$

\textbf{Acknowledgments.} The authors thank  Prof. A. Leiderman for his useful discussions. We also thanks the referees of the paper for their suggestions and improvements.

\end{document}